\documentclass[11pt]{article}

\usepackage[utf8]{inputenc}
\usepackage[T1]{fontenc}
\usepackage[top=2.5cm,bottom=2.5cm,left=2.5cm,right=2.5cm]{geometry}
\usepackage{amsmath,amssymb,amsthm}
\usepackage[colorlinks=true,linkcolor=red,citecolor=blue,urlcolor=magenta]{hyperref}
\allowdisplaybreaks[4]

\newtheorem{theorem}{Theorem}[section]
\newtheorem{lemma}[theorem]{Lemma}
\newtheorem{claim}[theorem]{Claim}
\newtheorem{corollary}[theorem]{Corollary}
\newtheorem{conjecture}[theorem]{Conjecture}
\newtheorem{remark}[theorem]{Remark}
\newcommand{\ignore}[1]{}

\begin{document}

\title{Tight connectivity and shadow densities in generalized Erd\H{o}s--Rogers problems}
\date{}

\author{
Lulu Dai\footnote{Center for Discrete Mathematics, Fuzhou University,
Fuzhou 350108, P.~R.~China. Email: \texttt{1415088965@qq.com}.}
\;\; and \;\;
Qizhong Lin\footnote{Center for Discrete Mathematics, Fuzhou University,
Fuzhou 350108, P.~R.~China. Email: \texttt{linqizhong@fzu.edu.cn}.
Supported in part by the National Key R\&D Program of China (Grant No.~2023YFA1010202)
and the NSFC (No.~12571361).}
}
\maketitle

\begin{abstract}
Let \(F\) and \(G\) be \(r\)-uniform hypergraphs, and let
\(f_{F,G}(n)\) be the largest integer \(m\) such that every \(n\)-vertex
\(G\)-free \(r\)-graph contains an induced \(F\)-free subgraph on \(m\)
vertices. We prove that, for \(r\ge3\) and \(2\le k\le r-1\), if \(F\) is
nonempty, \(G\) is \(k\)-tightly connected, and there is no homomorphism
from \(G\) to \(F\) (that is, \(G\not\to F\)), then
\[
        f_{F,G}(n)\le C(\log n)^{\beta_F^{(k)}},
        \qquad
        \beta_F^{(k)}=
        \max_{\emptyset\ne P\subseteq\partial_kF}
        \frac{e(P)}{v(P)-1}.
\]
The case \(r=3\) of our result resolves a conjecture of He and Nie.
As a consequence, we obtain the Ramsey lower bound
\(r(G,K_n^r)\ge2^{\Omega\bigl(n^{(r-1)/\binom rk}\bigr)}\)
for every \(k\)-tightly connected non-\(r\)-partite \(r\)-graph \(G\).
This extends a result of Conlon, Fox, Gunby, He, Mubayi, Suk, Verstra\"ete and Yu from the \(3\)-uniform setting.

\end{abstract}

\section{Introduction}

Given two \(r\)-uniform hypergraphs, or \(r\)-graphs, \(F\) and \(G\), let
\(f_{F,G}(n)\) be the largest integer \(m\) such that every \(n\)-vertex
\(G\)-free \(r\)-graph contains a set of \(m\) vertices whose induced
subgraph is \(F\)-free. The classical off-diagonal Ramsey problem is the
special case \(F=K_r^r\), where \(K_r^r\) denotes a single \(r\)-edge. In
this case \(f_{K_r^r,G}(n)\) is the minimum independence number among all
\(n\)-vertex \(G\)-free \(r\)-graphs and is inverse, in the usual sense, to
\(r(G,K_n^r)\).

For graphs, Erd\H{o}s and Rogers~\cite{ER} initiated the study of the more general
functions \(f_{K_s,K_t}(n)\), which are now known as Erd\H{o}s--Rogers
functions. These problems have attracted substantial attention. Among them, the adjacent
case \(t=s+1\) is particularly central and has a long history; it has subsequently been studied by
Bollob\'as and Hind~\cite{BollobasHind}, Krivelevich~\cite{Krivelevich94,Krivelevich95},
Alon and Krivelevich~\cite{AlonKrivelevich}, Dudek and R\"odl~\cite{DudekRodl},
Wolfovitz~\cite{Wolfovitz}, Dudek, Retter and R\"odl~\cite{DudekRetterRodl}, Dudek and Mubayi~\cite{DudekMubayi}, and Mubayi and Verstra\"ete~\cite{MubayiVerstraeteClassical}.
The case \(s=2\) belongs to classical off-diagonal Ramsey theory. The
results of Ajtai, Koml\'os and Szemer\'edi~\cite{AKS} and
Kim~\cite{Kim} give
$
        f_{K_2,K_3}(n)
        =
        \Theta\bigl(\sqrt{n\log n}\bigr).
$
Together with the work of Mattheus and
Verstra\"ete~\cite{MattheusVerstraete} for \(t=4\) and
Brada\v{c}~\cite{Bradac} for \(t\ge5\), these results show more generally
that, for every fixed \(t\ge3\),
$
        f_{K_2,K_t}(n)
        =
        n^{1/(t-1)+o(1)},
$
or equivalently,
\(r(K_t,K_m)=m^{t-1+o(1)}\).

For the adjacent cases with \(s\ge3\), Morris, Sahasrabudhe and
Verstra\"ete~\cite{MSV26} recently proved an upper bound which, together
with a lower bound that follows from a theorem of Joret, Micek, Reed and
Smid~\cite{JMRS}, yields
\[
        f_{K_s,K_{s+1}}(n)
        =
        \Theta\bigl(\sqrt{n\log n}\bigr)
\]
for every fixed \(s\ge3\).
Beyond the adjacent case, the first non-adjacent case \(t=s+2\) was
studied by Sudakov~\cite{Sudakov} and Janzer and
Sudakov~\cite{JanzerSudakov}. A further direction concerns arbitrary
graph pairs, which have also been studied extensively; see, for example,
\cite{MubayiVerstraeteArbitrary,BaloghChenLuo,NenadovSparseInduced,
GishbolinerJanzerSudakov}. Thus the graph case provides the classical
polynomial-scale benchmark against which the hypergraph problem should
be viewed.

For \(3\)-graphs, however, the behaviour changes dramatically. The
hypergraph analogue of the Erd\H{o}s--Rogers problem was first investigated
by Dudek and Mubayi~\cite{DudekMubayi}, initiating the systematic study of
these functions in uniform hypergraphs. The case \(s=3\) is special, since \(K_3^3\) consists of a single edge and hence \(f_{K_3^3,K_4^3}(n)\) is the minimum independence number of an \(n\)-vertex \(K_4^3\)-free \(3\)-graph. Equivalently, it is inverse to the off-diagonal Ramsey number \(r(K_4^3,K_m^3)\). Results of Conlon, Fox and Sudakov~\cite{ConlonFoxSudakovHypergraphRamsey} therefore give
\[
 c_1\left(\frac{\log n}{\log\log n}\right)^{1/2}
 \le f_{K_3^3,K_4^3}(n)
 \le c_2\frac{\log n}{\log\log n}.
\]

For every fixed \(s\ge4\), Dudek and Mubayi~\cite{DudekMubayi} proved
\[
 c_1(\log n)^{1/4}
 \left(\frac{\log\log n}{\log\log\log n}\right)^{1/2}
 \le f_{K_s^3,K_{s+1}^3}(n)
 \le c_2\log n.
\]
Conlon, Fox and Sudakov~\cite{ConlonFoxSudakovShortProofs} improved the lower
bound to \((\log n)^{1/3-o(1)}\), while Lin and Niu~\cite{LinNiu} recently
improved the upper bound to \(O(\log n/\log\log n)\). This transition from
polynomial to polylogarithmic behaviour becomes even more pronounced in
higher uniformities.

Indeed, iterated-logarithmic upper bounds appear in several adjacent
Erd\H{o}s--Rogers problems. Mubayi and Suk~\cite{MubayiSuk} proved that,
for every \(r\ge14\),
\[
        f_{K_{r+1}^r,K_{r+2}^r}(n)
        =
        O\bigl(\log_{(r-13)}n\bigr),
\]
where, as usual, \(\log_{(1)}x=\log x\) and
\(\log_{(j+1)}x=\log(\log_{(j)}x)\). Fan, Hu, Lin and
Lu~\cite{FanHuLinLu} subsequently improved this to
$
        f_{K_{r+1}^r,K_{r+2}^r}(n)
        =
        O\bigl(\log_{(r-3)}n\bigr)
$
for every \(r\ge5\).
In a different parameter range, Du, Hu, Liu and
Wang~\cite{DuHuLiuWangStep} proved that, for every fixed \(r\ge4\) and
every \(t\ge r+7\),
\[
        f_{K_{r+1}^r,K_t^r}(n)
        =
        \bigl(\log_{(r-2)}n\bigr)^{\Theta(1)}.
\]
They also obtained further related results in~\cite{DuHuLiuWangNote}.
More recently, Lin and Niu~\cite{LinNiu} proved, among other results, that
\[
        f_{K_s^4,K_{s+1}^4}(n)
        =
        (\log n)^{o(1)}
\]
for every fixed \(s\ge4\), while for every \(r\ge5\),
$
        f_{K_{r+1}^r,K_{r+2}^r}(n)
                =
        \bigl(\log_{(r-3)}n\bigr)^{o(1)}.
$
These developments reveal an increasingly sparse, iterated-logarithmic
regime for adjacent Erd\H{o}s--Rogers functions as the uniformity grows.

The preceding results concern highly structured clique-type pairs. To
understand the possible scales in the general Erd\H{o}s--Rogers setting, it
is useful to first consider a general lower bound. Since every independent
set is \(F\)-free whenever \(F\) is nonempty, we always have
\[
        f_{F,G}(n)
        \ge
        \max\bigl\{m:r(G,K_m^r)\le n\bigr\}.
\]
Thus the lower bound for \(f_{F,G}(n)\) is controlled by the inverse growth
of the off-diagonal Ramsey number \(r(G,K_m^r)\). In particular, the
standard upper bound of Erd\H{o}s and Rado \cite{ER52} implies that, for every fixed
\(r\)-graph \(G\) and every nonempty \(r\)-graph \(F\),
\[
        f_{F,G}(n)
        \ge
        \bigl(\log_{(r-2)}n\bigr)^{\Omega_G(1)}
        \qquad (r\ge4).
\]
On the other hand, if \(G\) has a single-exponential off-diagonal Ramsey
bound of the form
\[
        r(G,K_m^r)
        \le
        \exp\bigl(Cm^a(\log m)^b\bigr),
\]
then the above inversion gives
\[
        f_{F,G}(n)
        \ge
        c\left(
        \frac{\log n}{(\log\log n)^b}
        \right)^{1/a}.
\]
Hence, for such hypergraphs \(G\), the natural lower bound is already
polylogarithmic. More broadly, this illustrates the close connection
between the growth of generalized Erd\H{o}s--Rogers functions and the
problem of determining how structural properties of a fixed hypergraph
\(G\) govern the off-diagonal Ramsey number \(r(G,K_m^r)\), a direction
recently studied by Conlon, Fox, Gunby, He, Mubayi, Suk, Verstra\"ete and
Yu~\cite{ConlonFoxGunbyHeMubayiSukVerstraeteYu}.

These two possibilities illustrate that higher uniformity alone does not
determine the scale of an Erd\H{o}s--Rogers function. Indeed, such
polylogarithmic behaviour is not universal, even among the hypergraphs
considered in this paper. For example, when \(r=4\), the pair
\(F=K_4^4\) and \(G=K_5^4\) satisfies the structural assumptions below, but
the double-exponential lower bound for \(r(K_5^4,K_m^4)\) obtained by
Du, Hu, Liu and Wang~\cite{DuHuLiuWangDouble} and subsequently improved by
Fan, Li, Lin and Ning~\cite{FanLinDouble} yields
\[
        f_{K_4^4,K_5^4}(n)
        =
        O\bigl((\log\log n)^5\bigr).
\]
Thus, depending on the off-diagonal Ramsey growth of \(G\), both
polylogarithmic and iterated-logarithmic behaviours may occur. Determining
the precise scale for general pairs \((F,G)\) remains an important open
problem.

These observations motivate the search for structural parameters that
govern the general problem beyond clique-type pairs. Two natural notions
that arise in this context are tight connectivity and shadows. For
\(1\le k\le r-1\), an \(r\)-graph \(H\) is \emph{\(k\)-tightly connected}
if its edges can be ordered as \(e_1,\ldots,e_t\) so that, for every
\(i\ge2\), some \(j<i\) satisfies \(|e_i\cap e_j|\ge k\). The case
\(k=r-1\) is the usual notion of tight connectivity. A homomorphism from
\(G\) to \(F\) is a map \(\phi:V(G)\to V(F)\) such that
\(\phi(e)\in E(F)\) for every \(e\in E(G)\); we write \(G\not\to F\) if
no such map exists. The \(k\)-shadow \(\partial_kF\) is the \(k\)-graph
consisting of all \(k\)-sets contained in edges of \(F\). For
\(2\le k\le r-1\), define
\[
        \beta_F^{(k)}
        =
        \max_{\emptyset\ne P\subseteq\partial_kF}
        \frac{e(P)}{v(P)-1}.
\]
Here and throughout, the vertex set of a nonempty subgraph
\(P\subseteq\partial_kF\) is the union of its edges.

These notions are particularly natural in off-diagonal hypergraph Ramsey theory, where
connectivity often governs the transition between polynomial and superpolynomial growth. Conlon, Fox, Gunby, He, Mubayi, Suk, Verstra\"ete and
Yu~\cite{ConlonFoxGunbyHeMubayiSukVerstraeteYu} proved that every tightly
connected non-tripartite \(3\)-graph \(G\) satisfies
\(r(G,K_n^3)\ge2^{\Omega(n^{2/3})}\), equivalently
\(f_{K_3^3,G}(n)=O((\log n)^{3/2})\). This suggests that connectivity should
play a decisive role well beyond clique-type problems. He and
Nie~\cite{HN} developed this viewpoint for general pairs \(F,G\) and proved
the following polylogarithmic upper bound.

\begin{theorem}[He and Nie~\cite{HN}]\label{thm:HN}
Let \(r\ge3\), and let \(F\) and \(G\) be \(r\)-graphs. Suppose that
\(F\) is nonempty, \(G\) is \(2\)-tightly connected, and \(G\not\to F\).
Then there exists a constant \(C=C(F)\) such that, for all sufficiently
large \(n\),
\[
        f_{F,G}(n)\le C(\log n)^{\alpha_F},
        \qquad
        \alpha_F=
        \max_{\emptyset\ne P\subseteq\partial_2F}
        \frac{e(P)+1}{v(P)-1}.
\]
\end{theorem}

For \(3\)-graphs, He and Nie obtained the sharper exponent
\(e(\partial_2F)/(v(F)-1)\) under an additional density assumption on
\(\partial_2F\), and conjectured that the assumption is unnecessary in the
following stronger form.

\begin{conjecture}[He and Nie~\cite{HN}]\label{conj:HN}
Let \(F\) and \(G\) be \(3\)-graphs, with \(F\) nonempty. If \(G\) is tightly connected and
\(G\not\to F\), then there exists a constant \(C\) such that
\[
        f_{F,G}(n)\le C(\log n)^{\beta_F^{(2)}}.
\]
\end{conjecture}

He and Nie also obtained a higher-shadow bound under a stronger
homomorphism obstruction formulated at the level of the \(k\)-shadows. Our
result takes a different direction: it retains only the ordinary
homomorphism obstruction \(G\not\to F\), as in Theorem~\ref{thm:HN}, while
allowing the connectivity level of \(G\) and the shadow level of \(F\) to
vary together. Thus the \(k\)-tight connectivity of \(G\) is matched with
the density parameter of the \(k\)-shadow \(\partial_kF\), without imposing
any stronger homomorphism obstruction. This leads to our main result.

\begin{theorem}\label{thm:main}
Let \(r\ge3\) and \(2\le k\le r-1\), and let \(F\) and \(G\) be
\(r\)-graphs. Suppose that \(F\) is nonempty, \(G\) is \(k\)-tightly
connected, and \(G\not\to F\). Then there exist
\(C=C(F,k)>0\) and \(n_0=n_0(F,k)\in\mathbb N\) such that, for every
\(n\ge n_0\),
\[
        f_{F,G}(n)\le C(\log n)^{\beta_F^{(k)}},
        \qquad
        \beta_F^{(k)}=
        \max_{\emptyset\ne P\subseteq\partial_kF}
        \frac{e(P)}{v(P)-1}.
\]
\end{theorem}

For \(k=2\), Theorem~\ref{thm:main} removes the additional \(+1\) in the
numerator of the exponent in Theorem~\ref{thm:HN}; in particular, when
\(r=3\), it proves Conjecture~\ref{conj:HN}. For general \(r\ge4\), the
theorem gives a higher-shadow extension under the ordinary homomorphism
obstruction \(G\not\to F\): the \(k\)-tight connectivity of \(G\) is matched
with the density parameter of \(\partial_kF\). A further feature of the
theorem is that the constant \(C\) depends only on \(F\) and \(k\), and hence
the estimate, including the threshold \(n_0\), is uniform over all
\(k\)-tightly connected \(r\)-graphs \(G\) satisfying \(G\not\to F\).

\begin{remark}\label{rem:optimization}
If \(G\) is \(k\)-tightly connected, then it is \(j\)-tightly connected for
every \(2\le j\le k\). Applying Theorem~\ref{thm:main} at each such level
therefore gives
\[
        f_{F,G}(n)
        \le
        C(\log n)^{\min_{2\le j\le k}\beta_F^{(j)}},
\]
where \(C\) may depend on \(F\) and \(k\). 
\end{remark}

At each fixed connectivity level, the theorem also yields a transparent
Ramsey consequence.

\begin{corollary}\label{app}
Let \(r\ge3\) and \(2\le k\le r-1\). If \(G\) is a \(k\)-tightly connected
\(r\)-graph which is not \(r\)-partite, then
\[
        r(G,K_n^r)
        \ge
        2^{\Omega\left(n^{(r-1)/\binom rk}\right)}.
\]
\end{corollary}

For \(k=2\), this gives \(r(G,K_n^r)\ge2^{\Omega(n^{2/r})}\); when \(r=3\),
it recovers the lower bound of Conlon et
al.~\cite{ConlonFoxGunbyHeMubayiSukVerstraeteYu}. At the other extreme,
taking \(k=r-1\) gives \(r(G,K_n^r)\ge2^{\Omega(n^{(r-1)/r})}\) for every
tightly connected non-\(r\)-partite \(G\).

\medskip
\noindent{\em Proof of Corollary~\ref{app}.}
Apply Theorem~\ref{thm:main} with \(F=K_r^r\). Since
\(\partial_kK_r^r=K_r^k\), it remains only to compute the corresponding
shadow-density parameter. If \(\emptyset\ne Q\subseteq K_r^k\) and
\(q=v(Q)\), then \(q\ge k\) and \(e(Q)\le\binom qk\). Moreover,
\[
        \frac{\binom{q+1}{k}/q}{\binom qk/(q-1)}
        =
        \frac{(q+1)(q-1)}{q(q+1-k)}
        \ge1,
\]
so the function \(q\mapsto\binom qk/(q-1)\) is increasing for \(q\ge k\).
Consequently,
\[
        \max_{\emptyset\ne Q\subseteq K_r^k}
        \frac{e(Q)}{v(Q)-1}
        =\frac{\binom rk}{r-1}.
\]
Theorem~\ref{thm:main} therefore gives
\(f_{K_r^r,G}(N)\le C(\log N)^{\binom rk/(r-1)}\). As
\(G\not\to K_r^r\) is equivalent to \(G\) being non-\(r\)-partite, choose
\[
        N=
        \left\lfloor
        \exp\left(cn^{(r-1)/\binom rk}\right)
        \right\rfloor,
\]
where \(c>0\) is sufficiently small. Then
\(f_{K_r^r,G}(N)<n\) for all sufficiently large \(n\), so there is a
\(G\)-free \(r\)-graph on \(N\) vertices with independence number less than
\(n\). Consequently,
\(r(G,K_n^r)>N\), and hence
\(r(G,K_n^r)\ge2^{\Omega(n^{(r-1)/\binom rk})}\).
\hfill\(\Box\)

\begin{remark}\label{rem:sharpness}
It is natural to ask whether \(\beta_F^{(k)}\) is best possible. For
\(F=K_3^3\) and \(k=2\), Theorem~\ref{thm:main} gives
\(f_{K_3^3,G}(n)=O((\log n)^{3/2})\) for every tightly connected
non-tripartite \(3\)-graph \(G\). This agrees with the exponent \(3/2\)
obtained by inverting the lower bound
\(r(G,K_n^3)\ge2^{\Omega(n^{2/3})}\) of Conlon et
al.~\cite{ConlonFoxGunbyHeMubayiSukVerstraeteYu}, uniformly over this class
of forbidden \(3\)-graphs. For a fixed \(G\), however, the general bound
need not be sharp: if \(K_4^{3-}\) is obtained from \(K_4^3\)
by deleting one edge, then Fox and He~\cite{FoxHeForbiddenLink} proved
\[
        f_{K_3^3,K_4^{3-}}(n)
        =\Theta\left(\frac{\log n}{\log\log n}\right),
\]
whereas Theorem~\ref{thm:main} gives \(O((\log n)^{3/2})\). Determining for
which pairs \((F,G)\) the exponent \(\beta_F^{(k)}\) is best possible remains
an interesting open problem.
\end{remark}

\section{Suen's inequality}

We use the following form of Suen's inequality, introduced by
Suen~\cite{Suen}; the formulation below is Theorem~3 of
Janson~\cite{JansonSuen}.

Let \(\{A_i\}_{i\in I}\) be a finite family of events on a common
probability space. A simple graph \(\Gamma\) on vertex set \(I\) is called a
\emph{dependency graph} for this family if the following condition holds:
whenever \(S,T\subseteq I\) are disjoint and there is no edge of \(\Gamma\)
between \(S\) and \(T\), the two collections of events
\(\{A_i:i\in S\}\) and \(\{A_j:j\in T\}\) are independent. We write
\(i\sim j\) when \(\{i,j\}\in E(\Gamma)\).

\begin{lemma}[Janson~\cite{JansonSuen}]\label{lem:suen-min}
Let \(\{A_i\}_{i\in I}\) be a finite family of events, and let \(\Gamma\)
be a dependency graph for this family. Put
\[
        X=\sum_{i\in I}{\bf 1}_{A_i},
        \qquad
        \mu=\mathbb E [X]=\sum_{i\in I}\mathbb P(A_i),
\]
and define
\[
        \Delta=
        \sum_{\substack{\{i,j\}\subseteq I\\ i\sim j}}
        \mathbb P(A_i\cap A_j),
        \qquad
        \delta=
        \max_{i\in I}
        \sum_{\substack{j\in I\\ j\sim i}}
        \mathbb P(A_j),
\]
where the sum defining \(\Delta\) is over unordered adjacent pairs of
distinct indices. We use the convention that
\(\mu^2/(8\Delta)=+\infty\) when \(\Delta=0\), and
\(\mu/(6\delta)=+\infty\) when \(\delta=0\). Then
\[
        \mathbb P(X=0)
        \le
        \exp\left(
        -\min\left\{
        \frac{\mu^2}{8\Delta},
        \frac{\mu}{6\delta},
        \frac{\mu}{2}
        \right\}
        \right).
\]
\end{lemma}

\section{Proof of the Main Theorem}

We briefly describe the proof of Theorem~\ref{thm:main}. The argument is
based on a random-template construction that links the \(k\)-tight
connectivity of \(G\) with the density structure of the \(k\)-shadow of
\(F\). We independently color the \(k\)-sets of \([n]\) and associate with
each color an independent random map from \([n]\) to \(V(F)\). An \(r\)-set
is declared to be an edge when all its \(k\)-subsets receive the same color
and the corresponding map sends it to an edge of \(F\). If the resulting
\(r\)-graph contained a copy of \(G\), then \(k\)-tight connectivity would
propagate a single color throughout that copy, and the associated map would
give a homomorphism \(G\to F\). Thus the construction is always \(G\)-free.

To rule out large induced \(F\)-free subgraphs, we count labelled
transversal copies of \(F\) inside each fixed vertex set. The nontrivial overlaps arising from shared \(k\)-set variables are
encoded by nonempty subgraphs \(Q\subseteq\partial_kF\), and the inequality
\(e(Q)\le\beta_F^{(k)}(v(Q)-1)\) gives precisely the estimates needed for
the dependency terms in Suen's inequality. Taking
\(w=C(\log n)^{\beta_F^{(k)}}\), we obtain that a fixed \(w\)-vertex set is
\(F\)-free with probability at most \(\exp(-\Omega(w\log n))\); a union
bound over all such sets then completes the proof.

\medskip
\noindent{\em Proof of Theorem~\ref{thm:main}.}
We may assume that \(F\) has no isolated vertices. Indeed, if \(F^{-}\) is
obtained from \(F\) by deleting all isolated vertices, then
\(\partial_kF=\partial_kF^{-}\) and
\(\beta_F^{(k)}=\beta_{F^{-}}^{(k)}\). Moreover, \(G\not\to F\) implies
\(G\not\to F^{-}\), while any copy of \(F^{-}\) inside a vertex set of size
at least \(v(F)\) extends to a copy of \(F\) by assigning arbitrary unused
vertices to the isolated vertices of \(F\). Since the parameter \(w\)
introduced below tends to infinity with \(n\), we may assume
\(w\ge v(F)\), so this reduction is valid for every vertex set considered
in the proof.

Write \(V(F)=\{a_1,\ldots,a_s\}\), \(P=\partial_kF\), and \(m=e(P)\).
Since \(F\) has no isolated vertices, the \(k\)-graph \(P\) has vertex set
\(V(F)\).

Fix a constant \(C_0>1\). All constants implicit in the estimates below
that depend only on \(F\) and \(k\) are regarded as fixed before the auxiliary
parameters are chosen. We then choose \(c_1\) and \(c_2\) successively:
first \(c_1\) sufficiently large in terms of \(F\), \(k\), and \(C_0\), and
then \(c_2\) sufficiently large in terms of \(F\), \(k\), \(C_0\), and
\(c_1\). Put
\[
        \ell=\lceil c_1\log n\rceil,
        \qquad
        w=\left\lceil c_2(\log n)^{\beta_F^{(k)}}\right\rceil.
\]
Throughout the proof, constants denoted by \(c_F\) and \(C_F\) may change
from line to line and depend only on \(F\) and \(k\). Since only finitely
many such constants occur, the above order of choices makes the hierarchy
of constants unambiguous.

We shall construct an \(n\)-vertex \(G\)-free \(r\)-graph \(H\) such that
every \(w\)-vertex subset of \(V(H)\) contains a copy of \(F\).

Randomly color every \(k\)-set in \(\binom{[n]}k\) by a color in
\([\ell]\), independently and uniformly. Denote the resulting coloring by
\[\beta:\binom{[n]}k\to[\ell].\] For every color \(t\in[\ell]\),
independently choose a random map 
\[
        \gamma_t:[n]\to V(F),
\] where each
vertex is mapped uniformly to one of the \(s\) vertices of \(F\).

Define an \(r\)-graph \(H\) on vertex set \([n]\) as follows. An \(r\)-set
\(X\subseteq[n]\) is an edge of \(H\) if and only if there exists a color
\(t\in[\ell]\) such that
\[
        \beta(S)=t
        \quad\text{for every }S\in\binom Xk,
        \qquad\text{and}\qquad
        \gamma_t(X)\in E(F).
\]

\begin{claim}\label{claim:H-is-G-free}
The hypergraph \(H\) is \(G\)-free.
\end{claim}

\noindent{\em Proof.}
Suppose, for a contradiction, that \(H\) contains a copy of \(G\). For each
edge \(e\) of this copy, let \(c(e)\) be the common color of all members of
\(\binom ek\); this color is uniquely determined. By the \(k\)-tight
connectivity of \(G\), the edges can be ordered as \(e_1,\ldots,e_t\) so
that, for every \(i\ge2\), some \(j<i\) satisfies
\(|e_i\cap e_j|\ge k\). Any \(k\)-set contained in \(e_i\cap e_j\) then
forces \(c(e_i)=c(e_j)\). Inductively, all edges of the copy have one common
color, say \(t\). Hence the restriction of \(\gamma_t\) to the vertex set of this copy maps
every edge of \(G\) to an edge of \(F\), giving a homomorphism \(G\to F\).
This contradicts \(G\not\to F\), so \(H\) is \(G\)-free.
\hfill$\Box$

\medskip
It remains to show that, with positive probability, every vertex set of
size \(w\) contains a copy of \(F\). Fix a set \(W\subseteq[n]\) with
\(|W|=w\). 
Choose and fix a balanced partition
\[
        W=W_1\cup\cdots\cup W_s
\]
such that
\[
        |W_i|\ge \frac{w}{2s}
        \qquad\text{for every } i\in[s],
\]
which is possible for all sufficiently large \(n\). This
partition is deterministic and is used only to count labelled transversal
copies of \(F\).

Let
\[
        \mathcal I_W=
        \left\{
        (t,\mathbf{x}):
        t\in[\ell],\
        \mathbf{x}=(x_1,\ldots,x_s)\in
        W_1\times\cdots\times W_s
        \right\}.
\]
For each \((t,\mathbf{x})\in\mathcal I_W\), let \(A_{t,\mathbf{x}}\) be
the event that \(\mathbf{x}\) realizes a labelled copy of \(F\) with
color \(t\), according to the fixed labelling
\(V(F)=\{a_1,\ldots,a_s\}\). That is, \(A_{t,\mathbf{x}}\) is the event that
\[
        \beta(\{x_i:i\in I\})=t
        \quad\text{for every } I\in \binom{[s]}{k}  \quad\text{with} \quad \{a_i:i\in I\}\in E(P),
\]
and
\[
        \gamma_t(x_i)=a_i
        \quad\text{for every } i\in[s].
\]

If \(A_{t,\mathbf{x}}\) occurs, then \(H[W]\) contains a copy of \(F\) on
\(\{x_1,\ldots,x_s\}\). Indeed, for every \(e\in E(F)\), all \(k\)-subsets
of \(e\) belong to \(P=\partial_kF\), and hence all \(k\)-subsets of the
corresponding \(r\)-set have color \(t\); moreover, \(\gamma_t\) maps this
\(r\)-set onto \(e\).
Put \[X_W=\sum_{(t,\mathbf{x})\in\mathcal I_W}{\bf 1}_{A_{t,\mathbf{x}}}.\]
Restricting to these labelled transversal copies is sufficient: if \(H[W]\)
contains no copy of \(F\), then none of the events \(A_{t,\mathbf{x}}\)
occurs. Hence the probability that \(H[W]\) is \(F\)-free is at most
\(\mathbb P(X_W=0)\).

The following estimate is the main probabilistic part of the proof.

\begin{claim}\label{claim:zero-bound}
For every fixed \(W\in\binom{[n]}w\), we have
\[
        \mathbb P(X_W=0)
        \le
        \exp(-C_0w\log n).
\]
\end{claim}

We first show how Claim~\ref{claim:zero-bound} completes the proof. By the
union bound, Claim~\ref{claim:zero-bound} gives
\[
\begin{aligned}
        \mathbb P\left(
        \exists W\in\binom{[n]}w:
        H[W]\text{ contains no copy of }F
        \right)
        &\le
        \binom nw \exp(-C_0w\log n)  \\
        &\le
        \exp\bigl(-(C_0-1)w\log n\bigr)
        <1,
\end{aligned}
\]
where the second inequality uses
\(\binom nw\le\exp(w\log n)\). Therefore, with positive probability, every
\(w\)-vertex subset of \(H\) contains a copy of \(F\). Since \(H\) is
\(G\)-free by Claim~\ref{claim:H-is-G-free}, this gives
\(f_{F,G}(n)\le w\le C(\log n)^{\beta_F^{(k)}}\), as desired. It
remains to prove Claim~\ref{claim:zero-bound}.

\medskip
\noindent{\em Proof of Claim~\ref{claim:zero-bound}.}
Let \[\mu=\mathbb E[X_W].\] 
For fixed \((t,\mathbf{x})\in\mathcal I_W\), the event
\(A_{t,\mathbf{x}}\) imposes \(m=e(P)\) independent \(k\)-set color
conditions and \(s\) independent vertex-label conditions; here the
underlying random variables are distinct because
\(x_1,\ldots,x_s\) lie in pairwise disjoint parts. 
Hence
\[
        \mathbb P(A_{t,\mathbf{x}})
        =
        \ell^{-m}s^{-s}.
\]
Moreover,
\[
        |\mathcal I_W|
        =
        \ell\prod_{i=1}^s |W_i|
        \ge
        \ell\left(\frac{w}{2s}\right)^s.
\]
 Therefore
\begin{equation}\label{eq:mu}
        \mu
        =
        \sum_{(t,\mathbf{x})\in\mathcal I_W}
        \mathbb P(A_{t,\mathbf{x}})
        \ge
        \ell\left(\frac{w}{2s}\right)^s
        \ell^{-m}s^{-s}
        =
        c_Fw^s\ell^{1-m}.
\end{equation}

Since \(P=\partial_kF\) itself is one of the subgraphs in the definition
of \(\beta_F^{(k)}\), we have \(\beta_F^{(k)}\ge m/(s-1)\). For all
sufficiently large \(n\), \(w\ge c_2(\log n)^{\beta_F^{(k)}}\) and
\(\ell\le 2c_1\log n\).
Since \(1-m\le0\), the upper bound on \(\ell\) gives a lower bound for
\(\ell^{1-m}\). Thus \eqref{eq:mu} implies
\[
\begin{aligned}
        \mu
        &\ge
        c_Fw\cdot w^{s-1}\ell^{1-m}  \\
        &\ge
        c_Fw
        \bigl(c_2(\log n)^{\beta_F^{(k)}}\bigr)^{s-1}
        (2c_1\log n)^{1-m}  \\
        &\ge
        c_F2^{1-m}c_1^{1-m}c_2^{s-1}w\log n.
\end{aligned}
\]
By choosing \(c_2\) sufficiently large in terms of \(c_1\), \(C_0\),
\(F\), and \(k\), we obtain
\begin{equation}\label{eq:mu-large}
        \mu\ge 2C_0w\log n.
\end{equation}

We now apply Lemma~\ref{lem:suen-min}. Define a graph on the events
\(A_{t,\mathbf{x}}\) by joining two events whenever they share an underlying
random variable. Since the base variables
\(\{\beta(S):S\in\binom{[n]}k\}\) and
\(\{\gamma_t(v):t\in[\ell],\ v\in[n]\}\) are mutually independent, this is
a dependency graph. To describe its adjacency relation explicitly, for
\(\mathbf{x}=(x_1,\ldots,x_s)\) and
\(\mathbf{y}=(y_1,\ldots,y_s)\), write
\[R(\mathbf{x},\mathbf{y})=\{a_i\in V(P):x_i=y_i\}.\]
Thus \(R(\mathbf{x},\mathbf{y})\) records the vertices of the labelled copy
of \(F\) on which the two transversal tuples coincide. Since the sets
\(W_1,\ldots,W_s\) are pairwise disjoint, two events
\(A_{t,\mathbf{x}}\) and \(A_{u,\mathbf{y}}\) use a common \(k\)-set
color variable if and only if
\(e(P[R(\mathbf{x},\mathbf{y})])>0\). Equivalently, there exists an edge
\(\{a_i:i\in I\}\in E(P)\) such that \(x_i=y_i\) for every \(i\in I\).
In particular, if \(e(P[R(\mathbf{x},\mathbf{y})])>0\) and \(t\ne u\),
then \(\mathbb P(A_{t,\mathbf{x}}\cap A_{u,\mathbf{y}})=0\), since a common
\(k\)-set would be required to receive two distinct colors.

Thus two events are adjacent precisely when they share either a \(k\)-set
color variable or a vertex-label variable. Equivalently,
\begin{equation}\label{eq:adjacency}
\begin{split}
        A_{t,\mathbf{x}}\sim A_{u,\mathbf{y}}
        \quad\Longleftrightarrow\quad&
        e\bigl(P[R(\mathbf{x},\mathbf{y})]\bigr)>0\\
        &\text{or}\quad
        \bigl(
        R(\mathbf{x},\mathbf{y})\ne\emptyset
        \text{ and }t=u
        \bigr).
\end{split}
\end{equation}

Put
\[
        \Delta=
        \sum_{\substack{
        \{(t,\mathbf{x}),(u,\mathbf{y})\}\subseteq\mathcal I_W\\
        A_{t,\mathbf{x}}\sim A_{u,\mathbf{y}}
        }}
        \mathbb P(
        A_{t,\mathbf{x}}\cap A_{u,\mathbf{y}}
        ),
\]
and
\[
        \delta=
        \max_{(t,\mathbf{x})\in\mathcal I_W}
        \sum_{\substack{
        (u,\mathbf{y})\in\mathcal I_W\\
        A_{t,\mathbf{x}}\sim A_{u,\mathbf{y}}
        }}
        \mathbb P(A_{u,\mathbf{y}}).
\]

We first estimate \(\Delta\). If two adjacent events have different colors,
then by \eqref{eq:adjacency} they must share a \(k\)-set color variable. In
that case, the same \(k\)-set is required to receive two different colors,
and hence \(\mathbb P(A_{t,\mathbf{x}}\cap A_{u,\mathbf{y}})=0\). Thus only
same-color pairs contribute positively to \(\Delta\).

For \(1\le q\le s\) and \(0\le h\le m\), let \(\Delta_{q,h}\) be the
total contribution to \(\Delta\) from unordered pairs
\(\{(t,\mathbf{x}),(t,\mathbf{y})\}\) 
such that
\[
        |R(\mathbf{x},\mathbf{y})|=q,
        \qquad
        e(P[R(\mathbf{x},\mathbf{y})])=h.
\]
For such a pair, there are at most
\(C_F\ell w^{2s-q}\) choices. The two events together impose \(2m-h\)
distinct \(k\)-set color conditions. Indeed, because the parts
\(W_1,\ldots,W_s\) are pairwise disjoint, a \(k\)-set color variable
occurs in both events precisely when the corresponding edge of \(P\) is
contained in \(R(\mathbf{x},\mathbf{y})\). The two events also impose
\(2s-q\) distinct vertex-label conditions, and therefore
\[
        \mathbb P(A_{t,\mathbf{x}}\cap A_{t,\mathbf{y}})
        =
        \ell^{-(2m-h)}s^{-(2s-q)}.
\]
Consequently,
\[\Delta_{q,h}\le C_Fw^{2s-q}\ell^{1-2m+h}.\] Using the lower bound
\(\mu\ge c_Fw^s\ell^{1-m}\) from \eqref{eq:mu}, we have
\(w^{2s}\ell^{2-2m}\le c_F^{-2}\mu^2\). Consequently,
\begin{equation}\label{eq:Delta-type}
\begin{aligned}
        \Delta_{q,h}
        &\le
        C_Fw^{2s-q}\ell^{1-2m+h} \\
        &=
        C_F
        \frac{w^{2s}\ell^{2-2m}}
        {w^q\ell^{1-h}} \\
        &\le
        C_F\frac{\mu^2}{w^q\ell^{1-h}},
\end{aligned}
\end{equation}
after adjusting the constant \(C_F\).

We next show that, by choosing \(c_1\) and then \(c_2\) sufficiently large,
we have
\begin{equation}\label{eq:key-ratio}
        w^q\ell^{1-h}
        \ge
        8s(m+1)C_0C_Fw\log n
\end{equation}
for every pair \((q,h)\) that can occur.

If \(h=0\), then adjacency with the same color implies
\(R(\mathbf{x},\mathbf{y})\ne\emptyset\), and hence \(q\ge1\). Therefore
\(w^q\ell^{1-h}=w^q\ell\ge w\ell\ge c_1w\log n\).
Thus \eqref{eq:key-ratio} holds in this case after choosing \(c_1\)
sufficiently large in terms of \(F\), \(k\), and \(C_0\).

Now suppose \(h>0\). Then \(q\ge k\), and the common \(k\)-sets form the
nonempty subgraph \(P[R(\mathbf{x},\mathbf{y})]\) of \(\partial_kF\).
After deleting its isolated vertices, we obtain a nonempty
\(Q\subseteq\partial_kF\) with \(e(Q)=h\) and \(v(Q)\le q\). Hence, by the
definition of \(\beta_F^{(k)}\),
\[
        h=e(Q)
        \le
        \beta_F^{(k)}(v(Q)-1)
        \le
        \beta_F^{(k)}(q-1).
\]
Using \(w\ge c_2(\log n)^{\beta_F^{(k)}}\) and
\(\ell\le2c_1\log n\), we get
\[
\begin{aligned}
        w^q\ell^{1-h}
        &=
        w\ell\frac{w^{q-1}}{\ell^h} \\
        &\ge
        w\ell
        \frac{c_2^{q-1}}{(2c_1)^h}
        (\log n)^{\beta_F^{(k)}(q-1)-h} \\
        &\ge
        \frac{c_2^{q-1}}{(2c_1)^h}w\ell.
\end{aligned}
\]
Since \(q\ge k\ge2\), choosing \(c_2\) sufficiently large in terms of
\(c_1\), \(F\), \(k\), and \(C_0\) gives \eqref{eq:key-ratio}.

There are at most \(s(m+1)\) possible pairs \((q,h)\). Summing
\eqref{eq:Delta-type} over all of them and using \eqref{eq:key-ratio}, we
obtain
\[
        \Delta
        =
        \sum_{q,h}\Delta_{q,h}
        \le
        \frac{s(m+1)C_F\mu^2}
        {8s(m+1)C_0C_Fw\log n}
        =
        \frac{\mu^2}{8C_0w\log n}.
\]
If \(\Delta=0\), the next estimate follows from the convention in
Lemma~\ref{lem:suen-min}. Otherwise, dividing the preceding bound by
\(8\Delta\) gives
\begin{equation}\label{eq:mu2Delta}
        \frac{\mu^2}{8\Delta}
        \ge
        C_0w\log n.
\end{equation}

We next estimate \(\delta\). Fix an event \(A_{t,\mathbf{x}}\). We split
its neighbours into same-color and different-color neighbours.

First consider same-color neighbours. These have the form
\(A_{t,\mathbf{y}}\) with \(R(\mathbf{x},\mathbf{y})\ne\emptyset\).
There are at most \(C_Fw^{s-1}\) choices for \(\mathbf{y}\), and each
event has probability \(\ell^{-m}s^{-s}\). Hence their total contribution
to \(\delta\) is at most \(C_Fw^{s-1}\ell^{-m}\).

Now consider different-color neighbours. If \(u\ne t\), then the random
variables \(\gamma_t(v)\) and \(\gamma_u(v)\) are independent, even for
the same vertex \(v\). Thus different-color adjacency can only arise from
a shared \(k\)-set color variable. Consequently, at least \(k\) labelled
vertices must coincide. There are at most \(C_F\ell w^{s-k}\) such
different-color neighbours, and each has probability \(\ell^{-m}s^{-s}\).
Hence their total contribution is at most \(C_Fw^{s-k}\ell^{1-m}\).
Therefore
\begin{equation}\label{eq:delta}
        \delta
        \le
        C_F\left(
        w^{s-1}\ell^{-m}
        +
        w^{s-k}\ell^{1-m}
        \right).
\end{equation}

It remains to compare these two terms with \(\mu\). Since \(F\) is a
nonempty \(r\)-graph, the \(k\)-shadow \(\partial_kF\) contains a copy of
\(K_r^k\), and hence \(\beta_F^{(k)}\ge\binom rk/(r-1)\). Since
\(2\le k\le r-1\), we have \(\binom rk\ge r\), so
\((k-1)\beta_F^{(k)}\ge(k-1)r/(r-1)>1\). It follows from
\(w\ge c_2(\log n)^{\beta_F^{(k)}}\) that
\(w^{k-1}/\log n\to\infty\) as \(n\to\infty\).
Therefore, by choosing \(c_1\) and \(c_2\) sufficiently large, we may
ensure that, for all sufficiently large \(n\),
\[
        c_Fw\ell
        \ge
        12C_0C_Fw\log n,
        \qquad
        c_Fw^k
        \ge
        12C_0C_Fw\log n.
\]
Together with \eqref{eq:mu}, this gives
\[
\begin{aligned}
        \mu
        \ge
        c_Fw\ell\cdot w^{s-1}\ell^{-m} \ge
        12C_0w\log n\cdot
        C_Fw^{s-1}\ell^{-m},
\end{aligned}
\]
and
\[
\begin{aligned}
        \mu
        \ge
        c_Fw^k\cdot w^{s-k}\ell^{1-m}\ge
        12C_0w\log n\cdot
        C_Fw^{s-k}\ell^{1-m}.
\end{aligned}
\]
If \(\delta=0\), the next estimate again follows from the convention in
Lemma~\ref{lem:suen-min}. Otherwise, using \eqref{eq:delta} and the two
preceding bounds gives
\begin{equation}\label{eq:mudelta}
        \frac{\mu}{6\delta}
        \ge
        C_0w\log n.
\end{equation}

Finally, \eqref{eq:mu-large}, \eqref{eq:mu2Delta}, and
\eqref{eq:mudelta} give
\[
        \frac{\mu^2}{8\Delta}
        \ge
        C_0w\log n,
        \qquad
        \frac{\mu}{6\delta}
        \ge
        C_0w\log n,
        \qquad
        \frac{\mu}{2}
        \ge
        C_0w\log n.
\]
By Lemma~\ref{lem:suen-min},
\[
        \mathbb P(X_W=0)
        \le
        \exp\left(
        -\min\left\{
        \frac{\mu^2}{8\Delta},
        \frac{\mu}{6\delta},
        \frac{\mu}{2}
        \right\}
        \right)
        \le
        \exp(-C_0w\log n).
\]
This proves Claim~\ref{claim:zero-bound} and completes the proof of
Theorem~\ref{thm:main}.
\hfill$\Box$

\end{document}